\documentclass{article}
\usepackage{indentfirst}
\usepackage[tbtags]{mathtools}
\usepackage{amssymb}
\usepackage{amsmath}
\usepackage{amsthm,amscd}
\usepackage[top=1 truein,bottom=1.5 truein,hmargin={1 truein,1 truein}]{geometry}
\usepackage{url}
\usepackage[all,cmtip]{xy}
\usepackage{verbatim}
\mathtoolsset{showonlyrefs,showmanualtags}

\usetagform{default}
\makeatletter
\newtheorem*{rep@theorem}{\rep@title}
\newcommand{\newreptheorem}[2]{%
\newenvironment{rep#1}[1]{%
 \def\rep@title{#2 \ref{##1}}%
 \begin{rep@theorem}}%
 {\end{rep@theorem}}}
\makeatother

\newtheorem*{rep@proposition}{\rep@title}
\newcommand{\newrepproposition}[2]{%
\newenvironment{rep#1}[1]{%
 \def\rep@title{#2 \ref{##1}}%
 \begin{rep@proposition}}%
 {\end{rep@proposition}}}
\makeatother

\newtheorem{theorem}{Theorem}
\newtheorem{remark}{Remark}[section]
\newtheorem{Def}{Definition}
\newtheorem{lem}{Lemma}[section]
\newtheorem{prop}{Proposition}[section]

\newtheorem{cor}{Corollary}

\newtheorem{conj}{Conjecture}[section]
\newreptheorem{theorem}{Theorem}
\newreptheorem{proposition}{Proposition}

\newcommand{\N}{\mathbb{N}}
\newcommand{\Z}{\mathbb{Z}}
\newcommand{\Q}{\mathbb{Q}}
\newcommand{\R}{\mathbb{R}}
\newcommand{\C}{\mathbb{C}}
\newcommand{\orbit}[2]{\mathcal{O}_{#1}(#2)}
\newcommand{\A}{\alpha}
\newcommand{\B}{\beta}

\newcommand{\z}{\zeta}

\newcommand{\e}{\equiv}

\newcommand{\OK}{\mathcal{O}_{K}}

\newcommand{\la}{\langle}
\newcommand{\ra}{\rangle}

\DeclareMathOperator{\rad}{rad}
\DeclareMathOperator{\order}{o}
\DeclareMathOperator{\lcm}{lcm}
\DeclareMathOperator{\rank}{rank}

\title{ABC Implies There are Infinitely Many non-Fibonacci-Wieferich Primes - An Application of ABC Conjecture over Number Fields}
\author{Wayne Peng}
\date{\today}

\AtEndDocument{\bigskip{\footnotesize%
  \textsc{UR Mathematics, 915 Hylan Building, University of Rochester, RC Box 270138  Rochester, NY 14627} \par  
  \textit{E-mail address}: \texttt{jpeng4@ur.rochester.edu} \par
 
}}

\begin{document}

\maketitle	
\begin{abstract}
In this paper, we define $X$-base Fibonacci-Wieferich prime which is a generalized Wieferich prime where $X$ is a finite set of algebraic numbers. We are going to show that there are infinitely many non-$X$-base Fibonacci-Wieferich primes assuming the $abc$-conjecture of Masser-Oesterl\'{e}-Szpiro for number fields. We also provide a new conjecture concerning the rank of free part of abelian group generated by all elements in $X$, and we will use the arithmetic point of view and geometric point of view to give heuristic. 
\end{abstract}
\section{Introduction}
A degree $m$ linear recurrence sequence is a sequence $X=\{x_n\}_{\geq 0}$ defined by the recurrence relation
\[
x_{n+m}=c_0x_n+c_1x_{n+1}+c_2x_{n+2}+\cdots+c_{m-1}x_{n+m-1}\qquad\forall n\geq 0
\]
with initial values $x_0,x_1,\ldots, x_{m-1}\in \bar{Q}$ where $c_i$s are some given constants in $\bar{\Q}$. 

In this paper, we are aiming to study the periodic properties of a rational recurrence sequence modulo a proper integer. We should use the Fibonacci sequence $\{F_n\}_{n\geq 0}$ as our very first example. The Fibonacci number $F_n$ is defined by the simplest recurrence relation
\[
F_n=F_{n-1}+F_{n-2}\qquad\text{for }n\geq 2
\]
with initial values $F_0=0$ and $F_1=1$. Once we mod out  by $7$, we will get
\[
0,\ 1,\ 1,\ 2,\ 3,\ 5,\ 1,\ 6,\ 0,\ 6,\ 6,\ 5,\ 4,\ 2,\ 6,\ 1,\ 0,\ 1,\ \ldots.
\]
which is a periodic sequence of length $16$. The first research about this topic can be traced back to D.D. Wall\cite{Wall1960}. The following theorem given by Wall allows us to define the period function.
\begin{prop}
For every integer $m$, there exists some integer $n$ such that
\begin{equation}
F_n\e 0\mod m\qquad\mbox{and}\qquad F_{n+1}\e 1\mod m. \label{eq:2}
\end{equation}
Moreover, if we define $\pi(m)$ to be the smallest integer satisfying \eqref{eq:2}, then $n$ is divisible by $\pi(m)$.
\end{prop}
We will call $\pi(m)$ the period function of the Fibonacci sequence modulo $m$. Since $\pi(m)$ is so important in this paper, it deserves a definition.
\begin{Def}
The period function of the Fibonacci sequence modulo $m$, $\pi(m)$,  is the smallest integer satisfying \eqref{eq:2}.
\end{Def}
Wall conjectured that $\pi(p)\neq \pi(p^2)$ for all prime $p$. This is called ``Wall's conjecture", which is the kernel of this paper. Let us state it formally. 
\begin{conj}[Wall's Conjecture]\label{Wall's}
$\pi(p)\neq\pi(p^2)$ for all prime $p$.
\end{conj}
Our main purpose for this paper is to show that there are infinity many primes with $\pi(p)\neq \pi(p)^2$, and we can actually do more general.

We require to estate some notation for stating the first main result. I would like to call a given $2m$-tuple $(a_1,\ldots,a_m,b_1,\ldots, b_m)$ a \textbf{recurrence tuple} if $a_1,\ldots,a_m,b_1,\ldots,b_m\in\bar{\Q}\setminus\{0\}$ with $a_i\neq a_j$ for all $i\neq j$. I will also say that a sequence $X=\{x_n\}$ is a recurrence sequence generalized by the recurrence tuple if the closed form of the terms of the sequence is
\[
x_n=b_1a_1^n+b_2a_2^n+\cdots+b_ma_m^n\qquad\forall n\geq 0
\]
a number field $K$ is a splitting field of the tuple $(a_1,\ldots,a_m,b_1,\ldots,b_m)$ if $K$ is the splitting field of $a_1,\ldots,a_m,b_1,\ldots,b_m$ over $\Q$. I would also like to call a prime ideal $\mathfrak{p}$ $X$-\textbf{base Fibonacci-Wieferich prime} if the period of the sequence $X$ modulo $\mathfrak{p}$ is equal to one  modulo $\mathfrak{p}^2$, i.e. $\mathfrak{p}$ is a prime ideal such that Wall's conjecture fails. I will call an ideal $I$ of $K$ is \textbf{not degenerate relative to} the recurrence tuple $(a_1,\ldots,a_m,b_1,\ldots,b_m)$ if $v_{\mathfrak{p}_j}(a_i)=0$ and $v_{\mathfrak{p_j}}(b_i)=0$ for all $j$ and $i$ where $\mathfrak{p}_j$ is a prime ideal factor of $I$. If $I$ is not degenerate relative to the recurrence tuple, the order of $a_i\mod I^e$ for any integer $e$ is well-defined for all $i$. We denote the order by $\order_{I^e}(a_i)$. We also define
\[
\pi_X(I):=\text{the period of the sequence }\{x_n\mod I\}.
\]
With these, we are able to say the following proposition.
\begin{prop}\label{GNLS}
For any ideal $I=\mathfrak{p}_1^{e_1}\cdots\mathfrak{p}_r^{e_r}$ of $K$ not degenerate relative to the recurrence tuple $ (a_1,\ldots,a_m,b_1,\ldots,b_m)$, $\{x_n\mod I^e\}$ is periodic. Moreover, we have
\[
\pi_X(I^e)=\lcm\{\order_{\mathfrak{p}_i^{e_i}}(a_j)\mid \text{for all }i, j\}=\pi_X(\mathfrak{p}_1^{e_1})\pi_X(\mathfrak{p_2}^{e_2})\cdots \pi_X(\mathfrak{p}_r^{e_r}).
\]
\end{prop}

An \textbf{$a$-base Wieferich prime} by definition is a prime integer $p$ satisfying
\[
a^{p-1}\e1\mod p^2
\]
for a fixed integer $a$, or we can define more generally.

\begin{Def}
Let $K$ be an number field, and let $\A\in K$. An prime ideal $\mathfrak{p}$ of $K$ is called an $\A$-base Wieferich prime if 
\[
\A^{N_{K/\Q}(\mathfrak{p})-1}\e 1\mod\mathfrak{p}^2.
\]
\end{Def}
Thus, instead of considering $X$ as a sequence, Proposition~\ref{GNLS} suggests a generalized definition of Wieferich prime.
\begin{Def}
Let $X=\{a_1,a_2,\ldots,a_m\}$, and let $K$ be a number field containing all $a_i$ . An $X$-base Fibonacci-Wieferich prime $\mathfrak{p}$ is a prime ideal of $K$ with
\[
a_i^{N_{K/\Q}(\mathfrak{p})-1}\e 1\mod\mathfrak{p}^2\qquad\forall i
\]
where $N_{K/\Q}$ is the usual norm.
\end{Def}
As our first main theorem, we are going to show that there are infinitely many primes which are not $X$-base Fibonacci-Wieferich primes, assuming the $abc$-conjecture of Masser-Oesterl\'{e}-Szpiro for number fields. We shall call a number field $K$ $abc$-field if the $abc$-conjecture is true on $K$. 
\begin{theorem}\label{mainresult:1}
Let $K$ be the splitting field of $X=\{a_1,\ldots,a_m\}$. Then, there are infinitely many non-$X$-base Fibonacci-Wieferich primes, assuming $K$ is an $abc$-field.
\end{theorem}

We choose the recurrence tuple to be $((1+\sqrt{5})/2,(1-\sqrt{5})/2,1/\sqrt{5},-1/\sqrt{5})$, and Theorem~\ref{mainresult:1} deduces that there are infinitely many non-$F$-base Fibonacci-Wieferich primes (or non-Fibonacci-Wieferich primes for short), which are primes such that Wall's conjecture fails.

For obtaining a logarithm lower bound, we need some other notation. Let $K$ be any number field. Given an element $\gamma\in K$, We define 
\[
W_\gamma(B):=\{\text{prime ideal }\mathfrak{p}\mid N(\mathfrak{P})\leq \log B,\text{ and }\mathfrak{p}\text{ is an }X\text{-base Fibonacci-Wieferich prime}\}.
\]
Then, we follows Silverman's \cite{Silverman1988} closely to give a logarithm lower bound.
\begin{theorem}\label{mainresult:2}
For any non zero algebraic number $\gamma\in K\setminus\partial D_1(0)$ where $\partial D_1(0)$ is the complex unit circle with center at origin, there is a constant $C_{\gamma}$ such that we have
\[
|W_\gamma(B)|\geq C_{\gamma}\log B\qquad\forall B\gg 0
\]
assuming $K$ is an $abc$-field.
\end{theorem}

Heuristic result(\cite{GP2015}, \cite{CDP1997}, \cite{RE2007} or \cite{AJ2010}) indicates that we have the following conjecture, which contradicts to Wall's conjecture.
\begin{conj}\label{myconj}
Let $X=\{a_1,\ldots,a_m\}$. Then, we have
\begin{enumerate}
\item infinitely many $X$-base Fibonacci-Wieferich primes if the free rank of the multiplicative group $\la a_1,\ldots,a_m\ra$ is $1$.
\item finitely many $X$-base Fibonacci-Wieferich primes if the free rank of the multiplicative group $\la a_1,\ldots,a_m\ra$ is greater than $1$.
\end{enumerate}
\end{conj}

According to the fundamental theorem of abelian group, an abelian group can be written as the product of free part and torsion part, we define the free rank to be the rank of free part. The recurrence tuple of the Fibonacci sequence is $((1+\sqrt{5})/2,(1-\sqrt{5}/2),1/\sqrt{5},-1/\sqrt{5})$ Since the rank of the multiplicative group $\la (1+\sqrt{5})/2,(1-\sqrt{2})/2\ra$ is $1$, we would have infinitely many Fibonacci-Wieferich primes by assuming Conjecture~\ref{myconj} to be true. 

Geometric point of view on this problem could be explained by the terminology of dynamic systems. Let $V(K)$ be an variety over a number field $K$, and let $\phi:V(K)\to V(K)$ be a morphism of  $V$ into itself. For an integer $n\in \N$, denote by $\phi^n$ the $n$-th iterate $\phi\circ\cdots\circ\phi$; for any point $p\in V(K)$, we let $\orbit{\phi}{p}:=\{\phi^n(p)\}_n$ be a (forward) $\phi$-orbit of $p$. We are curious about what the relation between the sequence $\mathcal{O}_\phi(p)\mod\mathfrak{p}$ and $\mathcal{O}_\phi(p)\mod\mathfrak{p}^2$. We will talk more in the last section. 

Before Andrew Wile proved Fermat's Last Theorem, Wieferich primes and its generalizations are often connected to the first case of Fermat's Last Theorem. In 1909, Wieferich first showed that if a prime $p$ is the first case of Fermat's Last Theorem such that the Fermat's Last Theorem fails, then $p$ is a Wieferich prime\cite{Wieferich1909}. One years later, Mirimanoff shows that the same theorem also holds for base-$3$ Wieferich primes\cite{Mirimanoff1910}. Because of the hope of solving Fermat's Last Theorem, people had haven the motivation to know if the same theorem holds for any base-$m$ Wieferich primes. Some mathematicians, like Granville, Monagan\cite{Granville1988} and Suzuki\cite{Suzuki1994}, gave contributions. Zhi-Hong Sun and Zhi-Wei Sun\cite{SunSun1992} showed a similar theorem in 1992, which says that if Wall's conjecture is true, then the first case of Fermat's Last Theorem holds. So, Fibonacci-Wieferich primes are sometimes called \textbf{Wall-Sun-Sun primes} which are refer to potential counterexamples of Fermat's Last Theorem. The searching for Wall-Sun-Sun primes is still continuous, but nothing has been found up to $28\times 10^{15}$ reported by PrimeGrid\cite{PrimeGrid} in 2014. Although PrimeGrid has extended our searching to a such big number, it is not known if there are infinitely many primes for which \eqref{eq:2} is true or not unconditionally.

Silverman introduced a method to show that there are infinitely many non-Wieferich primes, assuming the traditional $abc$-conjecture, and our idea is to mimic Silverman's proof for our purpose. The $abc$-conjecture is very useful at dealing with this kind of questions. P. Ribenboim and G. Walsh\cite{PG1999} used the $abc$-conjecture to show that there are only finitely many powerful terms in the Lucas sequence and the Fibonacci sequence, and lately Minoru Yabuta\cite{Yabuta2007} further generalized this result.

I would like to briefly introduce Silverman's trick, and we should start with the traditional $abc$-conjecture.
\begin{conj}[The $abc$-conjecture on rational numbers]\label{traditionalabc}
if we have three integers $a,b,c\in \Z\setminus\{0\}$ which satisfy $a+b+c=0$, then for every $\varepsilon>0$ there is a constant $C_\varepsilon>0$ such that the following inequality holds
\[
\max\{|a|,|b|,|c|\}\leq C_\varepsilon rad(a,b,c)^{1+\varepsilon}
\]
where $rad(a,b,c):=\prod_{p|abc}p$. 
\end{conj}
This version is often called the $abc$-conjecture of Masser-Oesterl\'{e}-Szpiro. On the other hand, the $abc$-conjecture for function fields is a theorem of Stothers\cite{Stothers1981}. Note that Mason\cite{Mason1984} and Silverman\cite{Silverman1984} proved it several years later without known Stothers' result. In fact, the discovery of the $abc$-theorem for function fields is earlier than the $abc$-conjecture for rational numbers and then for number fields. 

The $abc$-conjecture is usually useful to solve arithmetic randomness question. Especially, it is useful to tell something about the squarefree part of a given sequence.
\begin{Def}
Let $N$ be an integer. We define $\kappa (N)$ to be the squarefree part of $N$, i.e.
\[
\kappa(N)=\prod_{p\parallel N}p.
\]
\end{Def}

For a sequence $X_n=a^n+b^n$ with $a,b\in\N$, we have, for $\varepsilon>0$,
\[
\max\{a^n,b^n,X_n\}\leq C\rad(a^nb^nX_n)^{1+\varepsilon}.
\]
The left hand side of the inequality increases as $n$ increases, so the right hand side should also increase. Note that $a^n$ and $b^n$ do not provide new prime factors, so the source of new primes comes from $X_n$. If we choose $\varepsilon$ sufficient small, we can see that the squarefree part cannot be bounded. Otherwise, the left hand side of the inequality will increase faster than the right hand side since we have $\rad(N^2)\geq N$ for any integer $N$.

The $abc$-conjecture over any arbitrary number field is a conclusion of Vojta's conjecture\cite{Vojta1998}. In 1998, Paul Vojta\cite{Vojta1998} formulated a new conjecture, and, as a consequence of this conjecture, the $abc$-conjecture and the $n$-tuple $abc$-conjecture on an arbitrary number field $K$ follows as special cases. One of many important achievement for Vojta's conjecture is that he generalized the $abc$-conjecture on rational numbers $\Q$ to number field $K$. Readers can also refer to \cite{GNT2013} which is an application of Vojta's $abc$-conjecture of universal form. 

We briefly list what we will do in each section. In Section 2, we will give an introduction of the $abc$-conjecture for number fields. Following, Section 3, we treat the most general cases and prove Theorem~\ref{mainresult:1}. Section 4 follows closely to \cite{Silverman1988} for getting a logarithm lower bound, Theorem~\ref{mainresult:2}. Finally, Section 5 talks about the heuristic of why the conjecture~\ref{myconj} is true from algebraic aspect and geometric aspect.

\section{The $abc$-conjecture for number fields}
We set the following:
\begin{enumerate}
\item $K$ is a number field, and let $\OK$ be the ring of integers of $K$ over $\Z$.
\item $\mathfrak{p}$ is a prime ideal of $K$, and let $I$ be an ideal of $K$.
\item $N_{K/\Q}(I):=|\OK/I|$. 
\item $G_K:=\{\sigma: K\to \C\mid \sigma\text{ is an embedding}\}$.
\item $V(f)$ is the zero set defined by the function $f$.
\item For an element $x\in K$, $N_{K/\Q}(x):=\prod_{\sigma\in G_K}\sigma(x)$.
\item We define $N(\cdot):=\log|N_{K/\Q}(\cdot)|/[K:\Q]$ where $[K:\Q]$ is the degree of $K$.
\item $M_K$ is the set of all finite and infinite places on $K$. For a finite place $v\in M_K$ and for a prime ideal $\mathfrak{p}$ which is correspondent to the place $v$, we denote $v$ by $v_{\mathfrak{p}}$ sometimes for clearing the correspondent relation, and we define
$$
\| x\|_{v_\mathfrak{p}}:=N_{K/\Q}(\mathfrak{p})^{-v_{\mathfrak{p}}(x)}\qquad\mbox{for every }x\in K.
$$
If $v\in M_K$ is a infinite place, then we define
$$
\| x\|_{v}:=|\sigma(x)|^e\qquad\text{with }|\cdot|\text{ the usual absolute value}
$$
for every non-conjugate embedding $\sigma\in G_K$ with $e=1$ if $\sigma$ is real, and $e=2$ if $\sigma$ is complex.
\item We define the absolutely logarithm height of any 3-tuples $(x_1,x_2,x_3)$ in $K^3\setminus \{0\}$ to be
$$
h(x_1,x_2,x_3):=\dfrac{1}{[K:\Q]}\sum_{v\in M_K}\log\max\{\| x_1\|_v,\| x_2\|_v,\| x_3\|_v\}.
$$
We also define the ideal of the point $(x_1,x_2,x_3)$ to be
$$
I_K(x_1,x_2,x_3):=\{\mathfrak{p}\subset\OK\mid \mbox{if }v_{\mathfrak{p}}(x_i)\neq v_{\mathfrak{p}}(x_j)\mbox{ for some }1\leq i,j\leq n\},
$$
and let
$$
\rad(x_1,x_2,x_3):=\dfrac{1}{[K:\Q]}\sum_{\mathfrak{p}\in I_K(x_1,x_2,x_3)}\log N_{K/\Q}(\mathfrak{p})=\sum_{\mathfrak{p}\in I_K(x_1,x_2,x_3)}\log N(\mathfrak{p}).
$$
\end{enumerate}
With all of these notation, the $abc$-conjecture says the following.
\begin{conj}\label{abc}
For any $\varepsilon>0$, there is a constant $C_{\varepsilon}$, depending only on $\varepsilon$, such that, for any tuple $(a,b,c)\in\mathbb{H}\setminus V(X)\cup V(Y)\cup V(Z)$ with $\mathbb{H}:=\{(X,Y,Z)\in \mathbb{P}^2(\bar{\Q})\mid X+Y+Z=0\}$, we have
$$
h(a,b,c)<(1+\varepsilon)\rad(a,b,c)+C_{\varepsilon}.
$$
\end{conj}
The $abc$-conjecture was first proposed by Joseph Oesterl\'{e} (1988), David Masser (1985) and Szpiro for the sake of creating an analogous theorem to Mason-Stothers theorem, which is the $abc$-theorem for function fields (see \cite{AT2002} for more historical introduction). The traditional $abc$-conjecture, Conjecture~\ref{traditionalabc}, is a special cases of Conjecture~\ref{abc}, which is a conclusion of Vojta's conjecture. A height function can be thought as a function to measure the arithmetic complexity of a given points on an affine, a projective space or a algebraic variety. With the language of height, the $abc$-conjecture says that the complexity of a given point on the hyperplane $X+Y+Z=0$ is bounded by those primes relative to the point, up to a constant and a scalar $1+\varepsilon$. 

Vojta's conjecture could be thought as an $abc$-conjecutre on algebraic geometry. A height function for a projective geometry depends on a divisor, and if we choose our divisor to be $[0]$, $[1]$ and $[\infty]$, Masser-Oesterl\'{e}-Szpiro conjecture follows as a special case. Vojta's conjecture also derives the $n$-tuples $abc$-conjecture for any number field. There is another version of $n$-tuple $abc$-conjecture for integers, which was proposed by J. Browkin and J. Brzezi\'{n}ski \cite{JJ1994}.


\section{There are infinitely many $X$-base Fibonacci-Wieferich primes}
We set the following notation for this section and for the following sections:
\begin{enumerate}
\item Let $(a_1,\ldots,a_m,b_1,\ldots,b_m)$ be a recurrence tuple. Without lose of generality, we may assume $a_i$ are all algebraic integers since this assumption does not affect the length of period.
\item Let $X=\{x_n\}_{n\geq 0}$ be a sequence generalized by the recurrence tuple $(a_1,\ldots,a_m,b_1,\ldots,b_m)$, i.e.
\[
x_n=b_1a_1^n+b_2a_2^n+\cdots b_ma_m^n\qquad\forall n\geq 0.
\]
\item We use $\order_{I^e}(\bar{a})$ to denote the order of $\bar{a}\in (\OK/I^e)^\times$ for a proper ideal $I$.
\item Let $\pi_X(I^e)$ be the length of the period of the sequence $\{x_n\mod I^e\}_{n\geq 0}$. Proposition~\ref{GNLS} shows that the $\pi_X$ is well-defined.
\end{enumerate}

The follows are some well-known factors about the notation we define above. First, The sequence $X$ is a recurrence sequence of degree $m$, i.e. $x_{n+m}=c_0x_n+c_1x_{n+1}+\cdots+c_{m-1}x_{n+m-1}$ with $c_i\in \bar{\Q}$ for all $n\geq 1$ with proper initial values. Second, an ideal $I$ of a number field $K$ where $K$ is the splitting field of the recurrence tuple $(a_1,\ldots,a_m,b_1,\ldots,b_m)$ could be expressed uniquely as a product of prime ideals $\mathfrak{p}_1^{e_1}\cdots\mathfrak{p}_r^{e_r}$.

Note that for all but finite many primes ideals are not degenerated relative to the recurrence tuple, and for those ideal $I$ not degenerate relative to the recurrence tuple, $\bar{a_1},\ldots, \bar{a_m}, \bar{b_1}\ldots,\bar{b_m}$ are unit in the ring $\OK/I^e\OK$ for any integer $e\geq 1$. Thus, we can consider the order of $a_i$ in the multiplicative group $(\OK/I^e\OK)^\times$. With all of these arguments, we are able to give the following theorem.
\begin{repproposition}{GNLS}
For any ideal $I=\mathfrak{p}_1^{e_1}\cdots\mathfrak{p}_r^{e_r}$ of $K$ not degenerate relative to the recurrence tuple $(a_1,\ldots,a_3,b_1,\ldots,b_m)$, $\{x_n\mod I^e\}$ is periodic. Then, we have
\[
\pi_X(I)=\lcm\{\order_{\mathfrak{p}_i^{e_i}}(a_j)\mid \text{for all }i, j\}=\pi_X(\mathfrak{p}_1^{e_1})\pi_X(\mathfrak{p_2}^{e_2})\cdots \pi_X(\mathfrak{p}_r^{e_r}).
\]
\end{repproposition}
\begin{proof}
By the Chinese Remainder Theorem, we can assume $I$ is the power of a prime ideal without lose of generality. Since the sequence $x_n=b_1a_1^n+\cdots b_ma_m^n$ is of degree $m$, the sequence $\{x_n\mod I\}$ is periodic if and only if the following system of equations holds.
\[
\begin{cases}
x_{\pi}&\e x_0\mod I^e\\
x_{\pi+1}&\e x_1\mod I\\
&\vdots\\
x_{\pi+m-1}&\e x_{m-1}\mod I
\end{cases}.
\]
It implies that
\[
\left[
\begin{matrix}
b_1 & b_2 & \cdots & b_m\\
b_1a_1 & b_2a_2 &\cdots & b_ma_m\\
\vdots & \vdots & \ddots & \vdots\\
b_1a_1^{m-1} &b_2a_2^{m-1} & \cdots & b_ma_m^{m-1}
\end{matrix}\right]
\begin{bmatrix}
a_1^{\pi}\\
a_2^{\pi}\\
\vdots\\
a_m^{\pi}
\end{bmatrix}
\e
\begin{bmatrix}
x_0\\
x_1\\
\vdots\\
x_{m-1}
\end{bmatrix}
\mod I
\]
where $\pi:=\pi_X(I)$. Since $I$ is not degenerate relative to $(a_1,\ldots,a_m,b_1,\ldots,b_m)$ and $a_i$s are all distinct, the matrix is invertible. Therefore, we have
\[
\begin{bmatrix}
a_1^{\pi}\\
a_2^{\pi}\\
\ldots\\
a_m^{\pi}
\end{bmatrix}
\e
\begin{bmatrix}
1\\
1\\
\vdots \\
1
\end{bmatrix}
\mod I.
\]
Hence, it is clear that $\pi\geq\lcm_i\{\order_{I}(a_i)\}$. The inequality for the other direction is trivial.
\end{proof}
A quick conclusion of this theorem is that the period should divide the order of the multiplicative group $(\OK/\mathfrak{p})^\times$ where $\mathfrak{p}$ is a unramified prime ideal, and this conclusion generalizes \cite{Wall1960}.
\begin{cor}\label{Cor:3}
Let $K$ be the splitting field of a rational tuple $(a_1,\ldots,a_m,b_1,\ldots,b_m)$, and let $\mathfrak{p}$ be a prime ideal of $K$ which is not degenerated relative to the tuple. Then, we have $\pi_X(\mathfrak{p})$ divides $N_{K/\Q}(\mathfrak{p})-1$. Let $p$ be an prime integer which is not degenerated relative to the tuple, and let $X$ be a rational sequence, i.e. every term $x_n$ is rational. Then, $\pi_X(p)$ divides $p^f-1$ where $f$ is the inertia degree of $p$ over $K$. If $p$ is ramified with ramified degree $e$, then $\pi_X(p)$ divides $p^{e-1}(p^f-1)$. 
\end{cor}
\begin{proof}
Obviously, the length of period is independent of the choice of fields, i.e. $\pi_X(\mathfrak{p}\OK)=\pi_X(\mathfrak{p}\mathcal{O}_L)$ for $K\subseteq L$. By the proposition~\ref{GNLS}, it trivially implies the first part of the corollary.

For the second part, since $K$ is Galois over $\Q$, the prime decomposition of $p\OK=\mathfrak{p}_1^{e}\cdots\mathfrak{p}_r^e$ where the inertia degrees of $\mathfrak{p}_i$ are same for all $i$. It implies
\[
(\OK/p\OK)^\times=(\OK/\mathfrak{p}_1^e)^\times\otimes\cdots\otimes (\OK/\mathfrak{p}_r^e)^\times,
\]
and it is clear that the multiplicative order of elements in $(\OK/p\OK)^\times$ divides the order of a component $(\OK/\mathfrak{p}_1^e)^{\times}$ which is $p^{e-1}(p^f-1)$ where $f$ is the inertia degree.
\end{proof}
Since for all but finitely many prime integers $p$ are not degenerate relative to the recurrence tuple, Corollary~\ref{Cor:3} holds for all but finitely many primes $p$. Proposition~\ref{GNLS} also indicates that $X$-base Fibonacci-Wieferich primes are generalized Wieferich primes. It generalizes in the sense that, instead of considering the order of a single element, we consider the order of a set of elements. Therefore, a prime which is not an $\A$-base Wieferich prime is also not an $X$-base Fibonacci-Wieferich prime where $\A$ is one of the generators of $X$. The following lemma proves this idea.
\begin{lem}~\label{fxxk}
Let $\mathfrak{p}$ be a prime which is not degenerate relative to the recurrence tuple. If we have
\[
a_i^n-1\in\mathfrak{p}\qquad\text{and}\qquad a_i^n-1\not\in\mathfrak{p}^2
\]
for some integer $n$ and $i$, then we have $\pi_X(\mathfrak{p})\neq\pi_X(\mathfrak{p}^2)$.
\end{lem}
\begin{proof}
It is equivalent to show that if $\pi:=\pi_X(\mathfrak{p})=\pi_X(\mathfrak{p}^2)$, then we will have $a_i^\pi-1\in\mathfrak{p}^2$ for all $n$ and $i$ satisfying $a_i^n-1\in\mathfrak{p}$ which is trivially true by Proposition~\ref{GNLS}.
\end{proof}
\begin{reptheorem}{mainresult:1}
Let $K$ be the splitting field of $X=\{a_1,\ldots,a_m\}$. Then, there are infinitely many non-$X$-base Fibonacci-Wieferich primes, assuming $K$ is an $abc$-field.
\end{reptheorem}
\begin{proof}
 Without lose of generality, we may also assume that $\A=a_i$ is in the ring of integers $\OK$. By lemma~\ref{fxxk}, it is sufficient to show that, for any $i$, there are infinitely many non-$a_i$-base Wieferich primes, or equivalently
\[
U_n:=\prod_{\mathclap{\substack{\mathfrak{p}\\ 
						  		\A^n-1\in\mathfrak{p}\\
                          		\A^n-1\not\in\mathfrak{p}^2}}}
\mathfrak{p}
\]
is unbounded. For the sake of contradiction, we assume $N(U_n)<B$ for some constant $B$, and also write $(\A^n-1)=U_nV_n$ for some ideal $V_n$. By the definition of the height function, we have a natural inequality
\[
N(U_nV_n)\leq h(\A^n,\A,(\A)^n-1).
\]
On the other hand, using Conjecture~\ref{abc}, for any $\varepsilon>0$, there exists a constant $C_{\varepsilon}$ such that
\[
h(\A^n,1,\A^n-1)\leq (1+\varepsilon)\rad(\A^n,1,\A^n-1)+C_{\varepsilon}.
\]
Note that
\[
\rad(\A^n,1,\A^n-1)=\sum_{\A(\A^n-1)\in\mathfrak{p}}N(\mathfrak{p})\leq N(\A)+N(U_n)+\dfrac{1}{2}N(V_n),
\]
so there exists a constant $C_{\varepsilon,\A}$ such that
\begin{equation}
\dfrac{(1-\varepsilon)N(V_n)}{2}\leq\varepsilon N(U_n)+C_{\varepsilon,\A}. \label{linlousu}
\end{equation}
However, for $\varepsilon<1$, \eqref{linlousu} could not hold once the $N(U_n)$ is bounded. Hence, we conclude that $U_n$ is unbounded.
\end{proof}
\section{A logarithm lower bound}
We define the following for this section and the follows:
\begin{enumerate}
\item Let $\gamma\in K$ with $|\gamma|\neq 1$.
\item The height function $h:K\to\R_{\geq 0}$ of any element $\gamma\in K$ is defined as
\[
h(\gamma):=\dfrac{1}{[K:\Q]}\sum_{v\in M_K}\log\max\{\|\gamma\|_v,1\}.
\]
\item The approximation function $\lambda_v:K\to\R_{\geq 0}$ with respect to a place $v$ is defined as, for every $\gamma$,
\[
\lambda_{v}(\gamma)=\dfrac{1}{[K:\Q]}\log\max\{\|\gamma\|_v,1\}.
\]
\item Let $M_K^{\infty}$ be the set of all infinity places of $K$.
\item $\varphi$ is the Euler totient function.
\item $\Phi_n(S)$ is the $n$-th cyclotomic polynomial.
\item Let $W_\gamma(B)=\{\text{prime ideal }\mathfrak{p}\mid N(\mathfrak{p})\leq \log B,\text{ and }\mathfrak{p}\text{ is not a } \gamma-\text{base Wieferich prime}\}$.
\item $m_\mathfrak{p}:=\order_\mathfrak{p}(\gamma)$ if $v_\mathfrak{p}(\gamma)=0$.
\end{enumerate}
To easy notation, we write
\begin{align*}
(\gamma^n-1)=\mathfrak{u}_n\mathfrak{v}_n\mathfrak{w}_n^{-1}\qquad&\text{with}\quad \mathfrak{u}_n=\kappa((\A^n-\B^n));\\
(\Phi_n(\gamma))=\mathcal{U}_n\mathcal{V}_n\mathcal{W}_n^{-1}\qquad&\text{with}\quad \mathcal{U}_n=\kappa((\A^n-\B^n)).
\end{align*}
By Theorem~\ref{GNLS} and Lemma~\ref{fxxk}, we have
\[
\{\text{prime ideal }\mathfrak{p}|\mathfrak{p} \text{ is not an }X-\text{ base Fibonacci-Wieferich Prime}\}=\bigcup_{i=1}^mW_{a_i}(\infty)
\]
with $X=\{a_1,\ldots,a_m\}$. Thus, an logarithm lower bound to $|W_{a_i}|$ for some $i$ is also a lower bound to non-$X$-base Fibonacci-Wieferich primes. Write $(\gamma)=IJ^{-1}$ where $I$ and $J$ are coprime ideals. We should note that the radical of the ideal $\mathfrak{w}_n$ and  the ideal $\mathcal{W}_n$ are both contained in the radical of $J$ since the poles of $(\gamma^n-1)$ and $(\Phi_n(\gamma))$ are the poles of $(\gamma)$. Since $W_\gamma(B)=W_{\gamma^{-1}(B)}$, we can further assume $\lambda_v(\gamma)>0$ for some $v\in M_K^{\infty}$ without lose of generality.
\begin{lem}\label{lm:4.1}
If $ (n)IJ+\mathfrak{p}=\OK$ and $\mathfrak{p}\supseteq \mathcal{U}_n$, then we have
\[
m_\mathfrak{p}=n\qquad\text{and}\qquad\gamma^{N(\mathfrak{p})-1}\not\e 1\mod\mathfrak{p}^2.
\]
\end{lem}
\begin{proof}
Since $\mathfrak{p}\supseteq\mathcal{U}_n$, we have $\Phi_n(\gamma)\e 0\mod\mathfrak{p}$, which also implies $\gamma^n\e 1\mod \mathfrak{p}$. Moreover, Since $(n)IJ+\mathfrak{p}=\OK$, the $S^n-1$ is separable modulo over $\mathfrak{p}$. Therefore, for all divisors $d$ of $n$, $\Phi_d(\gamma)\e 0\mod\mathfrak{p}$, i.e. $m_\mathfrak{p}=n$. 

Since $\mathfrak{p}^2\not\supseteq\mathcal{U}_n$, $\gamma^n\e 1+u\mathfrak{p}\mod\mathfrak{p}^2$ for some $u\in(\OK/\mathfrak{p})^\times$. It follows that
\[
\gamma^{N_{K/\Q}(\mathfrak{p})-1}\e(1+u\mathfrak{p})^{\frac{N_{K/\Q}(\mathfrak{p})-1}{n}}\e 1+\dfrac{N_{K/\Q}(\mathfrak{p})-1}{n}u\mathfrak{p}\mod\mathfrak{p}^2,
\]
which completes our proof.
\end{proof}
\begin{lem}\label{lm:4.2}
$|W_\gamma(B)|\geq|\{n\leq (\log B-\log 2)/h(\gamma)\mid N(\mathcal{U}_n)\geq nN(IJ)\}|$.
\end{lem}
\begin{proof}
Given an $n\leq (\log B-\log 2)/h(\gamma)$ with $N(\mathcal{U}_n)\geq nN(IJ)$, there exists a prime ideal $\mathfrak{p}_n\supseteq\mathcal{U}_n$ with $nIJ+\mathfrak{p}_n=1$ , so we have
\[
m_{\mathfrak{p}_n}=n\qquad\text{and}\qquad\gamma^{N_{K/\Q}(\mathfrak{p}_n)-1}\not\e 1\mod\mathfrak{p}_n^2
\]
by Lemma~\ref{lm:4.1}.

Since $n\leq(\log B-\log 2)/h(\gamma)$, we have 
\[
|N(\mathfrak{p}_n)|\leq h(\gamma^n-1)\leq nh(\gamma)+\log 2\leq \log B.
\]
Moreover, if we found $\mathfrak{p}_n=\mathfrak{p}_n'$, then
\[
n=m_{\mathfrak{p}_n}=m_{\mathfrak{p}_n'}=n'
\]
which shows that, for every integer $n$ with $N(\mathcal{U}_n)\geq nN(IJ)$, we can construct a unique prime ideal in $W_{\gamma}$.
\end{proof}

We also need to generalize Lemma 5 of \cite{Silverman1988}, which shows that we can find an absolutely constant $c>0$ such that $\Phi_n(a,b)\geq e^{c\varphi(n)}$ for all $n\geq 2$. The detail of Silverman's proof can be found in a lecture note \cite{Silverman1987} published by Springer-Verlag. In our case, $n\geq 2$ will be weaken to $n$ large enough, but an weaker lemma would not affect our final result because we only care about $n$ large enough.

This lemma is an easy application of equidistribution of primitive roots of unity. In general, if we let $f:\C\to\C$ be a continuous function, then 
\[
\dfrac{\sum_{(i,n)=1}f(\z_n^i)}{\phi(n)}\to \int_{\partial D_{1}(0)}f(z)dz
\]
where $\partial D_{1}(0)$ is the complex unit circle. Readers can find more details in \cite{BIR2008}.
\begin{lem}\label{lm:5}
We can find a constant $c>0$ such that
$$
N(\Phi_n(\gamma))\geq c\varphi(n).
$$
for all $n\gg 0$.
\end{lem}
\begin{proof}
Since the primitive $n$-th roots of unity $\zeta_n^i$ for $(i,n)=1$ are equidistribution on the unit circle, we have
\[
N(\Phi_n(\gamma))=\dfrac{1}{[K:\Q]}\sum_{\sigma\in G_K}\sum_{(i,n)=1}\dfrac{\log (\zeta_n^i-\sigma(\gamma))}{\phi(n)}\to\dfrac{1}{[K:\Q]}\sum_{\sigma\in G_K}\int_{\partial D_1(0)}\log(z-\sigma(\gamma))dz
\]
when $n$ goes to infinity. By Jensen's formula, it implies
\[
\dfrac{N(\Phi_n(\gamma))}{\phi(n)}\to\sum_{v\in M_K^{\infty}}\lambda_v(\gamma)\qquad\text{as}\qquad n\to\infty,
\]
so $N(\Phi_n(\gamma))/\phi(n)$ greater than zero for some large enough $n$ which completes our proof.
\end{proof}
\begin{lem}\label{lm:7}
Fix $\delta>0$. Then
$$
|\{n\leq Y\mid \phi(n)\geq\delta n\}|\geq (\dfrac{6}{\pi^2}-\delta)Y+O(\log Y).
$$
(The big-O constant is absolute.)
\end{lem}
\begin{proof}
See \cite{Silverman1988}.
\end{proof}
\begin{lem}\label{lm:4.5}
If we assume $K$ is an $abc$-field, then, for all $\varepsilon>0$, there exists a constant $C_{\varepsilon,\gamma}$ satisfying 
\[
N(\mathcal{V}_n)\leq n\varepsilon h(\gamma)+C_{\varepsilon,\gamma}.
\]
\end{lem}
\begin{proof}
Since $\mathfrak{v}_n\subseteq\mathcal{V}_n$, we only need to prove a similar estimate for $\mathfrak{v}_n$. Given $\varepsilon>0$, by Conjecture~\ref{abc}, there exists a constant $C_{\varepsilon}$ satisfying
\[
h(\gamma^n,1,\gamma^n-1)\leq (1+\varepsilon)\rad(\gamma^n,1,\gamma^n-1)+C_{\varepsilon}.
\]
On the left hand side, we have
\[
h(\gamma^n-1)\leq h(\gamma^n,1,\gamma^n-1).
\]
On the right hand side, we have
\[
(1+\varepsilon)\rad(\gamma^n,1,\gamma^n-1)\leq (1+\varepsilon)(N(IJ)+N(\mathfrak{u}_n)+\dfrac{N(\mathfrak{v}_n)}{2}).
\]
Because of $N(\mathfrak{u}_n)+N(\mathfrak{v}_n)-N(\mathfrak{w}_n)=\sum_{\sigma\in G_K}\log|\sigma(\gamma)^n-1|/[K:\Q]\leq h(\gamma^n-1)$, and we put the above two inequalities together, we have
\[
h(\gamma^n-1)\leq (1+\varepsilon)\left(N(IJ)+h(\gamma^n-1)+N(\mathfrak{w}_n)-\dfrac{N(\mathfrak{v}_n)}{2}\right)+C_{\varepsilon}.
\]
Therefore, we have
\begin{align*}
N(\mathfrak{v}_n)&\leq\dfrac{2\varepsilon}{1+\varepsilon}h(\gamma^n-1)+2(N(IJ)+N(\mathfrak{w}_n))+\dfrac{2}{(1+\varepsilon)}
C_{\varepsilon}\\ 
&\leq \dfrac{2\varepsilon n}{1+\varepsilon}h(\gamma))+\dfrac{2\varepsilon\log 2}{1+\varepsilon}+2(N(IJ)+N(\mathfrak{w}_n))+\dfrac{2}{(1+\varepsilon)}
C_{\varepsilon}.
\end{align*}
Replacing $C_{\varepsilon}$ by $C_{\varepsilon,\gamma}:=2(\varepsilon\log 2+C_{\varepsilon})/(1+\varepsilon)+2(N(IJ\mathfrak{w}_n))$, and replacing $\varepsilon$ by $\varepsilon/(2-\varepsilon)$, we have desire result.
\end{proof}
\begin{reptheorem}{mainresult:2}
For any non zero algebraic number $\gamma\in K\setminus\partial D_1(0)$ where $\partial D_1(0)$ is the complex unit circle with center at origin, there is a constant $C_{\gamma}$ such that we have
\[
|W_\gamma(B)|\geq C_{\gamma}\log B\qquad\forall B\gg 0
\]
assuming $K$ is an $abc$-field.
\end{reptheorem}
\begin{proof}
Let $c$ be an absolutely constant obtained from Lemma~\ref{lm:5}, and let $c_1,c_2,\ldots$ be constant not depended on $n$. By the Lemma~\ref{lm:4.2}, an integer $n\leq (\log B-\log 2)/h(\gamma)$ with $N(\mathcal{U}_n)>nN(IJ)$ can construct a unique prime $p_n$ in $W_\gamma(B)$. Thus, the lower bound for $|\{n\leq (\log B-\log 2)/h(\gamma)\mid N(\mathcal{U}_n)\geq nN(IJ)\}|$ is also a lower bound for $|W_\gamma(B)|$. We use Lemma~\ref{lm:4.5} and Lemma~\ref{lm:5} to estimate $N(\mathcal{U}_n)$. Let $C_{\varepsilon,\gamma}$ be the constant given in Lemma~\ref{lm:4.5}. Note that
\[
n\varepsilon h(\gamma)+N(\mathcal{U}_n)+C_{\varepsilon,\gamma}\geq N(\mathcal{V}_n)+N(\mathcal{U}_n)=|N(\Phi_n(\gamma))|,
\]
so we have
\begin{equation}\label{eq:4}
|N(\mathcal{U}_n)|\geq c\varphi(n)-n\varepsilon h(\gamma)-C_{\varepsilon,\gamma}
\end{equation}
for large enough $n$. Hence, if
\begin{equation}\label{eq:5}
c\varphi(n)-n\varepsilon h(\gamma)-c_2\geq nN(IJ)
\end{equation}
holds, then $N(\mathcal{U}_n)$ is greater than $nN(IJ)$. Thus, we have
\begin{equation}
|W_u(B)|\geq|\{n\leq (\log B-\log 2)/h(\gamma)\mid \mbox{\eqref{eq:5} holds.}\}|.
\end{equation}
\eqref{eq:5} is equivalent to
\[
c\varphi(n)-n\varepsilon h(\gamma)-c_2-nN(IJ)\geq 0.
\]
Now, we give some arbitrary fixed $\delta>0$, and let $\varepsilon=c\delta/2h(\gamma)$, and further suppose that $n$ satisfies $\varphi(n)\geq \delta n$. Then,
\begin{equation}\label{eq:6}
c\varphi(n)-n\varepsilon h(\gamma)-c_2- nN(IJ)\geq \dfrac{1}{2}c\delta n-c_2-n N(IJ),
\end{equation}
and there exists some integer $n_0:=n_0(\delta,c)$ only depended on $\delta$ and $c$ such that the right hand side of \eqref{eq:6}  is positive whenever $n\geq n_0$. Consequently, for those $n$ satisfying $n\geq n_0$ and $\phi(n)\geq \delta n$ simultaneously, it will also satisfy \eqref{eq:5}. Therefore, combining Lemma~\ref{lm:7}, we have
\begin{align*}
|W_u(B)|&\geq|\{n_0\leq n\leq (\log B-\log 2)/h(\gamma)\mid\varphi(n)\geq\delta n\}|\\
&\geq (\dfrac{6}{\pi^2}-\delta)(\log B-\log 2)/h(\gamma)+O(\log(\log B-\log 2)/h(\gamma))-n_0.
\end{align*}

Since we are free to choice $\delta>0$, the proof is completed.
\end{proof}
\begin{cor}
$|\{p\leq B\mid p\text{ is a Fibonacci-Wieferich prime}\}|\geq O(\log B)$.
\end{cor}

\section{Heuristic result and the conjecture}
\begin{enumerate}
\item Let $G_m:=K^\times=K\setminus\{0\}$, and let $G_m^n$ be $m$ folder multiplicative group where the product of two vectors is just the coordinates product.
\item Let $\la a_1,\ldots,a_m\ra$ be a multiplicative group generated by $a_i\in\bar{\Q}$, i.e. an element of the group could be written as a finite product $\prod_{i=1}^m a_i^{e_i}$ where $e_i$s are some integers.
\item By the fundamental theorem of abelian group, a multiplicative group $G$ is isomorphic to $\Z/m\Z\times\Z^r$ for some $m$ and $r$, where $\Z/m$ is the torsion part of $G$ and $\Z^r=\Z\times\Z\times\cdots\times\Z$ is the free part of $G$. We define the rank of $G$, $\rank G$, to be $r$.
\item Let $V$ be a smooth variety. Let $\dim V$ be the dimension of the variety $V$.
\item Let $\phi:V\to V$ be an smooth endomorphism, and let $q$ be a point in $V$. The orbit $\mathcal{O}_\phi(q)$ of $q$ under the map $\phi$ is a sequence 
\[
\{\phi^n(q)\}_{n\geq 0}
\]
where we use $\phi^n$ to represent that $\phi$ iterates $n$ times.
\end{enumerate}
I would like to separate this section into two subsections. Both subsections are for explaining the same thing. One of the subsections is from the arithmetic point of view to make the heuristic argument of Conjecture~\ref{myconj}, and the other is from the geometric point of view.
\subsection{Arithmetic point of view}
Fermat's little theorem says that $2^{p-1}\e 1\mod p$ for every prime $p$, so it must be that $2^{p-1}\e 1+k_pp\mod p^2$ for some $0\leq k_p\leq p-1$. If we are to assume that $k_p$ is distributed randomly between $0$ and $p-1$, the possibility of $k_p\e0\mod p$ would be $1/p$. Therefore, the expected number of Wieferich primes below $Y$ is given by
\[
\sum_{p\leq Y}1/p\approx \log\log Y.
\]
This is the well-known heuristic argument that why people conjugated that the number of Wieferich primes is infinite even though we only find few of them.

Given a number field $K$ and an element $\A\in K$, since $\A^{N(\mathfrak{p})-1}\e 1\mod\mathfrak{p}$, we must have
\[
\A^{N(\mathfrak{p})-1}\e 1+k_{\mathfrak{p}}\mathfrak{p}\mod\mathfrak{p}^2
\]
where $k_\mathfrak{p}\in\OK/\mathfrak{p}$. Supposing $k_\mathfrak{p}$ distributed randomly, we would have the possibility $1/N_{K/\Q}(\mathfrak{p})$ for getting $k_\mathfrak{p}\e 0\mod \mathfrak{p}$. Therefor, the expected number of $\A$-base Wieferich primes with norm below $Y$ is given by
\[
\sum_{\mathfrak{p};N_{K/\Q}(\mathfrak{p})\leq Y}\dfrac{1}{N_{K/\Q}(\mathfrak{p})},
\]
which is also tend to infinity as $Y$ goes to infinity.

Given a recurrence tuple $(a_1,\ldots,a_m,b_1\ldots,b_m)$, since $\pi_X(\mathfrak{p}^e)=\lcm_{1\leq i\leq m}\{\order_{\mathfrak{p}^e}(a_i)\}$, $\pi_X(\mathfrak{p})=\pi_X(\mathfrak{p}^2)$ happens if $k_{\mathfrak{p}}\e 0\mod\mathfrak{p}$ simultaneously for each $a_i$. We should also consider the case as $a_i$'s are not multiplicative dependent. For example, Let $\A$ be a quadratic unit, and let $\bar{\A}$ is its conjugates. Since $|\A\bar{\A}|=1$, the order of $\A$ modulo $\mathfrak{p}$ is depended on the other. Thus, the expected number of $X$-base Fibonacci-Wieferich primes with norm under $Y$ is given by
\[
\sum_{\mathfrak{p};N_{K/\Q}(\mathfrak{p})\leq Y}\dfrac{1}{N_{K/\Q}(\mathfrak{p})^r}
\]
where $r=\rank\la a_1,\ldots,a_m\ra$.
\subsection{Geometric points of view}
A recurrence sequence can be considered as a dynamic system. Let $q_i:=(x_i,x_{i+1},\ldots,x_{i+m})$ be a point on $\bar{\Q}^m$. Then, the matrix
\[
M:=
\left[\begin{matrix}
0 & 1 & 0 & \cdots & 0\\
0 & 0 & 1 & \cdots & 0\\
\ldots &\ldots &\ldots &\ddots &\ldots\\
0 & 0 & 0 & \cdots & 1\\
c_1 & c_2 & c_3 & \cdots & c_m
\end{matrix}\right]
\]
is a linear transformation for which $M(q_i)=q_{i+1}$ for all $i\geq 0$.  We should note that the length of the period of the sequence $X$ mod $\mathfrak{p}^e$ is equal to the length of the period of the orbit $\mathcal{O}_M(q_0)$ over modulo $\mathfrak{p}^e$. If we assume the Zarisky closure $\overline{\{q_i\}_{i\geq 0}}$ is of dimension $d$, it means that we can freely choice $d$ many coordinates, and the other coordinates are depended. Then, we want to ask the following question: if the $k_\mathfrak{p}$ is random distributed, what is the possibility of $\pi=\pi_X(p)=\pi_X(p^2)$? Equivalently, we want to know the possibility for that the congruence
\[
M^\pi(q_0)\e q_0\mod \mathfrak{p}^2
\]
holds. We definitely has
\[
M^\pi(q_0)\e \begin{pmatrix}
x_0+k_{0,\mathfrak{p}}\mathfrak{p}\\
x_1+k_{1,\mathfrak{p}}\mathfrak{p}\\
\vdots\\
x_{m-1}+k_{m-1,\mathfrak{p}}\mathfrak{p}
\end{pmatrix}
\mod \mathfrak{p}^2
\]
by the definition of $\pi$. Since the point is on a variety of dimension $d$, and there are $N_{K/\Q}(\mathfrak{p})$ many choices of $k_{i,\mathfrak{p}}$ for each $i$, the possibility of $k_{0,\mathfrak{p}}\e k_{1,\mathfrak{p}}\e\cdots\e k_{m-1,\mathfrak{p}}\e 0\mod \mathfrak{p}$ is $1/N_{K/\Q}(\mathfrak{p})^d$. Therefore, as long as the $\dim V=1$, we expect that there are infinitely many $X$-base Fibonacci-Wieferich primes.

The rank of the multiplicative group $\la a_1,\ldots,a_m\ra$ should be equal to the dimension of  $\overline{\{q_i\}_{i\geq 0}}$, and we are going to prove this as the end of this paper.

Our main idea is the following. Since the eigenvalues of our matrix $M$ are distinct, $M$ is diagonalizable. We can find an invertible matrix $B$ such that $M=B^{-1}AB$. Let $r_i=B(q_i)$. This implies $A(r_i)=r_{i+1}$, so we have a commutative diagram
\[
\begin{CD} 
q_i @>B>> r_i\\ 
@VVMV @VVAV\\ 
q_{i+1} @>B>> r_i+1 
\end{CD} .
\]
It is clear that the length of the period of the orbit $\mathcal{O}_M(q_0)$ modulo $\mathfrak{p}^e$ is equal to one of the orbit $\mathcal{O}_A(r_0)$ modulo $\mathfrak{p}^e$ for every prime $\mathfrak{p}$ and integer $e$. Moreover, if  $\prod_{j\in I} a_j=1$ or $-1$ where $I$ is some index set, then we have
\[
\prod_{j\in I}r_i^{(i)}=\prod_{j\in I} a_j^i\prod_{j\in I}r_0^{(j)}
\]
where we use $r^{(j)}$ to represent the $j$-th coordinate of a point $r$, i.e. $r_i$ belongs to the variety defined by the equation $\prod_{j\in I}Y_j=\pm \prod_{j\in I}r_0^{(j)}$ where $Y_i$s are variables.

We need two lemmas from \cite{laurent_equations_1984} and \cite{TY2005}.
\begin{lem}[Laurant]\label{Laurant}
Let $V\subseteq G_m^n$ be a subvariety. Let $\Gamma$ be a multiplicative subgroup of $G_m^n$. If $V\cap \Gamma$ is Zarisky dense in $V$, then $V$ is a subgroup of $G_m^n$.
\end{lem}
\begin{proof}
See \cite{laurent_equations_1984}.
\end{proof}
The following lemma is a well-known result in Representation theory, and we only present the part we need.
\begin{lem}\label{final}
If a subvariety $V$ is a proper subgroup of $G_m^n$, then $V$ contains in the zero set of the polynomial $X_1^{e_1}\cdots X_n^{e_n}-1$ for some $e_i\in\Z$.
\end{lem}
\begin{proof}
Let $\Gamma$ be the functor from the category of varieties over $K$ to the finite $K$-algebra. Then, the regular functions of $G_m^n$ is
\[
\Gamma(G_m^n)=K[Y_1^{\pm},Y_2^{\pm},\cdots, Y_n^{\pm}]=\bigoplus_{(e_1,\ldots,e_n)\in\Z^n}KY_1^{e_1}\cdots Y_n^{e_n}.
\]

Let $X(V)=\{\text{morphism }\chi:V\to G_m\mid \chi\text{ is a group homomorphism}\}$, and we claim every elements in $X(V)$ is linear independent. Suppose it is not the case for the sake of contradiction, and then we can find $\chi_1,\ldots,\chi_n\in X(V)$ pairwise distinct linearly dependent elements with $n$ minimal for the property. Then, there exist $\lambda_1,\ldots,\lambda_{n-1}\in G_m$ such that
\[
\chi_n=\lambda_1\chi_1+\cdots\lambda_{n-1}\chi_{n-1}.
\]
Since $\chi_n\neq \chi_1$, we have $\chi_n(\A)\neq\chi_1(\A)$ for some $\A\in V$. Thus, for any $\B\in V$, we have
\[
\chi_n(\A)\chi_n(\B)=\lambda_1\chi_1(\A)\chi_1(\B)+\cdots+\lambda_{n-1}\chi_{n-1}(\A)\chi_{n-1}(\B)=\chi_n(\A)(\lambda_1\chi_1(\B)+\cdots+\lambda_{n-1}\chi_{n-1}(\B)).
\]
However, it implies
\[
\lambda_1(\chi_1(\A)-\chi_n(\A))\chi_1(\B)+\cdots\lambda_{n-1}(\chi_{n-1}(\A)-\chi_n(\A))\chi_{n-1}=0
\]
which is a contradiction to the assumption of $n$.

The morphism $\chi$ is called character, and it is well known that $\chi(X_1,\ldots,X_n)$ is of the form $X_1^{e_1}\cdots X_n^{e_n}$ for some $e_i\in\Z$. Considering the ideal of $V$, denoted by $I(V)$, we want to show that $I(V)=(\chi_1-1,\chi_2-1,\ldots,\chi_k-1)$ for some $\chi_i\in X(G_m^n)$. Given $f\in I(V)\setminus\{0\}$, we can find $\chi_1,\ldots,\chi_m\in X(G_m^n)$ pairwise distinct and $\lambda_1,\ldots,\lambda_m\in K$ such that
\[
f=\lambda_1\chi_1+\cdots\lambda_m\chi_m.
\]
Note that $\chi_i$s are naturally linearly independent. After reindexing the $\chi_i$ if necessary, we may find $1=i_1<i_2<\cdots<i_{l+1}=m+1$ such that, for any $i_j\leq r,s<i_{j+1}$, $ \chi_r\big|_V=\chi_s\big|_V$. Let $\theta_j:=\chi_{i_j}\big|_V$ where $\theta_j$s are pairwise linearly independent, and we then have
\[
0=f\big|_V=\sum_{j=1}^l(\sum_{i_j\leq i<i_{j+1}}\lambda_i)\theta_j.
\]
Hence, it follows that $\sum_{i_j\leq i<i_{j+1}}\lambda_i=0$ for all $j=1,\ldots,l$, or, equivalently,
\[
\lambda_{i_j}=-\sum_{i_j<i<i_{j+1}}\lambda_i.
\]
We set
\[
\mu_j:=\sum_{i_j<i<i_{j+1}} \lambda_i\chi_i\chi^{-1}_{i_j}-1,
\]
where we should note that $\chi_i\chi^{-1}_{i_j}\in X(G_m^n)$, so $f=\mu_1\chi_{i_1}+\cdots\mu_l\chi_{i_l}$ is generated by $\chi-1$ for some $\chi\in X(G_m^n)$. Since $\Gamma(G_m^n)$ is Noetherian, there exists $\chi_1,\ldots\chi_k$ such that $I(V)=(\chi_1-1,\ldots,\chi_k-1)$.
\end{proof}
\begin{prop}
$\rank\la a_1,a_2,\ldots,a_m\ra=\dim\overline{\{q_i\}}$.
\end{prop}
\begin{proof}
We would like to first assume that the $a_i$s are multiplicative independent, i.e. $\rank\la a_1,a_2,\ldots,a_m \ra=m$.

Let $A$ be the diagonal matrix for the matrix $M$, and we know
\[
A=
\left[
\begin{matrix}
a_1 & 0   & \cdots & 0\\
0   & a_2 & \cdots & 0\\
\ldots &\ldots & \ddots &\ldots\\
0 & 0 & \cdots & a_m
\end{matrix}\right].
\]
We then consider $A$ as a vector $(a_1,a_2,\ldots,a_m)\in G_m^m$ since all $a_i$s are not zero. Thus, $\{r_i\}=\{A^ir_0\}$ can be interpreted as the orbit of $r_0$ under the action of $\{A^i\}$ in the group $G_m^m$. Note that $r_0^{-1}$ is belong to $G_m^m$ since all $b_i$s are not zero.

Let $W$ be a nontrivial subvariety of $K^m$ such that $A^kr_0\in W$ for infinitely many $k$, and it is equivalent to say that $A^k\in r_0^{-1}W$ for infinitely many $k$ where the transferor variety is still an variety. Since $a_i$s are not root of unity, $\{A^k\}$ contains infinitely many distinct points; therefore, $r_0^{-1}W$ is a subgroup of $G_m^m$ by Lemma~\ref{Laurant}. However, since we assume $a_i$ are multiplicative independent, the above conclusion contradicts to Lemma~\ref{final}. Hence, $W$ is $K^m$ since $W$ has more than one element.

If $\la a_1,\ldots,a_m\ra$ are multiplicative dependent, we let $W$ be the subvariety defined by all multiplicative relations, and follow the same argument which also implies $W=\overline{\{q_i\}}$.
\end{proof}

For the abstract, this whole paper are some concrete cases of the following question. Let $K$ be an number field, $V$ be an variety over $K$, and $\phi:V\to V$ be an endormorphism. Given an initial point $q$ and a prime ideal $\mathfrak{p}$ of $K$, we want to know whether the length of the periodic cycle of the sequence $\{\phi^n(q)\mod \mathfrak{p}\}_{n\geq 0}$ is equal to the one of the sequence $\{\phi^n(q)\mod \mathfrak{p}^2\}_{n\geq 0}$ or not. The arithmetic information hides in the map $\phi$, and we should expect that there are infinitely many primes $\mathfrak{p}$ having different length of periods because of the $abc$-conjecture. However, for some certain map, the $abc$-conjecture could be useless. For example, if the map is ramified, then every factor of $\phi^n(p)$ will not be squarefree.
\bibliography{ABCFIB}
\bibliographystyle{Annual}
\end{document}